\documentclass[12pt]{article}
	\usepackage{pslatex}
	\usepackage{fancyhdr}
	\usepackage{graphicx}
	\usepackage{geometry}
	\RequirePackage[latin1]{inputenc} \RequirePackage[T1]{fontenc}

	\def\figurename{Figure} 
	\makeatletter
	\renewcommand{\fnum@figure}[1]{\figurename~\thefigure.}
	\makeatother
	
	\def\tablename{Table} 
	\makeatletter
	\renewcommand{\fnum@table}[1]{\tablename~\thetable.}
	\makeatother
	\usepackage{color}
	[2000/03/29 v1.0m Standard LaTeX package]
	
	\usepackage{amsmath}
	\usepackage{amssymb}
	\usepackage{amsfonts}
	\usepackage{amsthm,amscd}

	\newtheorem{theorem}{Theorem}[section]

	\newtheorem{proposition}{Proposition}[section]

	\theoremstyle{remark}
	\newtheorem{remark}[theorem]{Remark}
	
	\numberwithin{equation}{section}

	\def\D{\mathbb D}
	\def\P{\mathbb P}
	\def\R{\mathbb R}
	\def\E{\mathbb E}

	\def\Q{\mathbb Q}

	\def\E{\mathbb E}

\begin{document}
		
\title{New approach to optimal control of delayed stochastic Volterra integral equations}
		
\author{Roméo Kouassi Konan \thanks{romeokouadiokonan071@gmail.com}\;\; and \; Auguste Aman \thanks{aman.auguste@ufhb.edu.ci/augusteaman5@yahoo.fr, corresponding author}\\
UFR Math\'{e}matiques et Informatique, Université Félix H. Boigny, Cocody,\;\;\;\;\;\;\;\;\;\;\\}
		
\date{}
\maketitle \thispagestyle{empty} \setcounter{page}{1}
		
\thispagestyle{fancy} \fancyhead{}
\fancyfoot{}
\renewcommand{\headrulewidth}{0pt}
		
\begin{abstract}
We address the optimal control of stochastic Volterra integral equations with delay through the lens of Hida-Malliavin calculus. We show that the corresponding adjoint processes satisfy an anticipated backward stochastic Volterra integral equation (ABSVIE), and, exploiting this structure, we establish both necessary and sufficient stochastic maximum principles. Our results provide a comprehensive and rigorous framework for characterizing optimal controls in delayed stochastic systems.
\end{abstract}
	
\vspace{.08in} \noindent \textbf{MSC}:Primary: 60F05, 60H15; Secondary: 60J30\\
\vspace{.08in} \noindent \textbf{Keywords}: Backward stochastic Volterra integral equations, time delayed generators, Hölder continuity condition.
			
\section{Introduction}
Optimal control of stochastic systems has been a central topic in applied mathematics, with applications spanning finance, engineering, and biology. In particular, stochastic Volterra integral equations (SVIEs) with delay arise naturally in models where the evolution of the system depends not only on its current state but also on its past history. Delays and memory effects introduce significant mathematical challenges, making the analysis and control of such systems a rich and active area of research.
Traditional approaches for controlling stochastic differential equations (SDEs) often rely on backward stochastic differential equations (BSDEs) and the classical stochastic maximum principle. However, for SVIEs with delay, the non-Markovian nature of the system prevents a straightforward application of these classical tools. Recent developments have highlighted the role of anticipated BSDEs, which naturally accommodate the forward-looking dependence on future values of the adjoint processes, as well as Hida-Malliavin calculus, which provides powerful techniques for handling anticipative stochastic integrals.
In this work, we propose a novel approach to the optimal control of stochastic Volterra integral equations with delay. By exploiting the duality between stochastic delay differential equations (SDDEs) and anticipated backward stochastic differential equations (ABSDES), we derive a rigorous framework to characterize optimal controls. We establish both necessary and sufficient conditions for optimality through stochastic maximum principles, providing theoretical tools that extend the classical results to the delayed, non-Markovian setting. Our approach opens new avenues for the analysis and computation of optimal strategies in systems with memory and delay effects.
More precisely, in this paper we consider a controlled stochastic Volterra-type integral equation with delay of the form:
\begin{eqnarray}\label{ee1}
\left\{
\begin{array}{l}
\displaystyle X^{u}(t)= x_{0}(t)+ \int_{0}^{t}b(t,s,X^{u}(s-\delta),u(s)) ds\\\\\;\;\;\;\;\;\;\;\;\;\;\;\;\;\;\;\;\;\;\;\;\;\;\;\;\; +\displaystyle\int_{0}^{t} \sigma(t,s,X^{u}(s-\delta),u(s))dB(s), \; 	t \in [0,T],\\\\
X^{u}(t) = x_{0}(t), \quad t \in [-\delta, 0],
\end{array}\right.
\end{eqnarray}
where $x_{0},\,b$ and $\sigma$ are some given function and $u$ denotes the control process, assumed to take its values in a given set $\mathcal{U}$ and $B$ is a Brownian motion defined on a filtered probability space $\left( \Omega, \mathcal{F},  \left\lbrace \mathcal{F}_{t} \right\rbrace_{t \geqslant 0}, \mathbb{P} \right)$ satisfying the usual conditions. 

For a given function $f$ and $g$, let us consider the performance functional $J$ defined by: for all $u\in \mathcal{U}$,
\begin{eqnarray*}
J(x_0,u) = \mathbb{E} \left[ \int_0^T f(t, X^{u}(t-\delta), u(t)) \, dt + g(X^{u}(T))|X^{u}(0)=x_{0}(0) \right].
\end{eqnarray*}

The main motivation of this paper is to establish an existence of a control $u^{*}\in \mathcal{U}$ called an optimal control that maximizes the performance functional $J$ such that
\begin{eqnarray}\label{CP}
	\phi(x_{0})=J(x_0,u^*)=\sup_{u\in \mathcal{U}} J(x_0,u).
\end{eqnarray}
Several variants of this problem have already been investigated in the framework of optimal control of stochastic Volterra integral equations (SVIEs), notably by Yong \cite{yong2006,yong2008ptf} and by Agram \emph{et al.} \cite{Agram2015,Agram2018}. However, to the best of our knowledge, despite these contributions, the corresponding approach has not yet been extended to delayed systems in the context of stochastic Volterra differential equations. One may also mention the works of Bernt \O ksendal, Agnès Sulem, and Tusheng Zhang \cite{Elsanousi2000}, who analyzed optimal control with delay in the setting of stochastic differential equations with memory, driven by both a Brownian motion and a Poisson random measure. Furthermore, in the contributions of Nacira Agram, Bernt \O ksendal, and Samia Yakhlef \cite{Agram2018}, the control problem is reformulated within the framework of stochastic Volterra equations.
  
Our approach in this paper differs from the aforementioned works. In particular, the presence of the terms $X(s-\delta)$ renders the problem non-Markovian, thereby precluding the use of a finite-dimensional dynamic programming approach. Nevertheless, we demonstrate that it is still possible to derive a Pontryagin-Bismut-Bensoussan type stochastic maximum principle.

The paper is organized as follows. Section 2 introduces the formulation of the control problem and presents some preliminary results. In Section 3, we derive the Hamiltonian function together with the corresponding adjoint equation associated with our control problem. Finally, Sections 4 and 5 are devoted respectively to the establishment of a sufficient and a necessary stochastic maximum principle.

\section{Formulation of problem}
Let $\left( \Omega, \mathcal{F}, \{ \mathcal{F}_t \}_{t \ge 0}, \mathbb{P} \right)$ be a filtered probability space equipped with a Brownian motion $(B(t))_{t \in [0, T]}$ satisfying the usual conditions and $\mathcal{U}$ designed an admissible control set.
For $u\in \mathcal{U}$, we recall the controlled state $(X^{u}(t))_{t\geq 0}$, described by a stochastic Volterra integral equation with delay
\begin{eqnarray}\label{A1}
\left\{
\begin{array}{l}
\displaystyle X^{u}(t)= x_{0}(t)+ \int_{0}^{t}b(t,s,X^{u}(s-\delta),u(s)) ds\\\\\;\;\;\;\;\;\;\;\;\;\;\;\;\;\;\;\;\;\;\;\;\;\;\;\;\; +\displaystyle\int_{0}^{t} \sigma(t,s,X^{u}(s-\delta),u(s))dB(s), \; 	t \in [0,T],\\\\
X^{u}(t) = x_{0}(t), \quad t \in [-\delta, 0],
\end{array}\right.
\end{eqnarray}

and the cost functional
\begin{eqnarray}\label{cost}
	J(x_0,u) = \mathbb{E} \left[ \int_0^T f(t, X^{u}(t-\delta), u(t)) \, dt + g(X^{u}(T))|(X^{u}_s)_{s\in[-\delta,0]}=x_0 \right].
\end{eqnarray}
Throughout this work, we will use the following spaces: 
\begin{description}
\item $\bullet$ $\mathcal{S}^{2}(\R)$ is the set $\R$-valued $\{ \mathcal{F}_t \}_{t \ge 0}$-adapted continuous process $(\varphi(t))_{t\geq 0}$ such that
\begin{eqnarray*}
	\E\left(\sup_{0\leq t\leq T}|\varphi(t)|^2\right)<+\infty.
\end{eqnarray*}
\item $\bullet$  $\mathcal{M}^{2}(\R)$ is the set $\R$-valued $\{ \mathcal{F}_t \}_{t \ge 0}$-adapted continuous process $(\varphi(t))_{t\geq 0}$ such that
\begin{eqnarray*}
	\E\left(\int^{T}_{0}|\varphi(t)|^2dt\right)<+\infty.
\end{eqnarray*}
\end{description}
To conclude this section, we recall some notions and results related to the Hida-Malliavin derivative, which, according to our approach, is very important for deriving the adjoint equation and the related results. Indeed, let suppose $u^{*}$ be a optimal control and denote $u^{\varepsilon}=u^{*}+\varepsilon u$, for $u$ a given admissible control. We consider the process $Y$ defined by 
\begin{eqnarray*}
	Y(t)=\frac{dX^{u^{\varepsilon}}(t)}{d\varepsilon}|_{\varepsilon=0},
\end{eqnarray*}
where $X^{u^{\varepsilon}}$ is the solution  of \eqref{A1} with control $u^{\varepsilon}$. Unlike the case of classical stochastic differential equations (SDEs), when the state equation is driven by SVIE \eqref{A1}, it is not possible to directly isolate the quantity $Y$. To do it, like in \cite{Agram2018}, we need a the help of Hida-Malliavin derivative. Let briefly give notion some information of the notion.

Let first consider simple random variables of the form
\begin{eqnarray*}
	F=f\left(\int_0^Th_1(t)dB(t),\cdots,\int_0^T h_n(t)dB(t)\right),
\end{eqnarray*}
where $f\in C^{\infty}(\R^{n})$ and $h_i\in L^2([0,T])$ for $i=1,\cdots, n$. The Malliavin derivative $D_tF$ is then defined by:
\begin{eqnarray*}
	D_tF=\sum_{i=1}^{n}\frac{\partial f}{\partial x_i}\left(\int_0^T h_1(t)dB(t),\cdots,\int_0^T h_n(t)dB(t)\right)h_i(t),\;\; t\in [0,T].
\end{eqnarray*}
Thus, $D_tF$ is a stochastic process (in $t$), representing the sensitivity of $F$ to a small perturbation of the noise $B$ at time $t$. he operator $D$ is closable in $L^2(\Omega)$ which allows one to define a stochastic Sobolev space:
\begin{eqnarray*}
	\D^{1,2}=\{F\in L^{2}(\Omega),\;\; DF\in L^{2}([0,T]\times\Omega)\}. 
\end{eqnarray*}
One of the fundamental results is the Clark-Ocone duality formula
\begin{proposition}\label{GCO}
\textbf{(Generalized Clark--Ocone Formula~\cite{Aase2000})}\\
For all \( F \in L^2(\mathcal{F}_T, \mathbb{P}) \), we have:
\begin{eqnarray*}
F = \mathbb{E}[F] + \int_0^T \mathbb{E}[D_t F \mid \mathcal{F}_t] \, dB(t).
\end{eqnarray*}
\end{proposition}

Moreover, the following generalized duality formula holds for Brownian motion.
\begin{proposition}
	\textbf{(Generalized Duality Formula for the Brownian motion \( B \))}\\
	Fix \( s \in [0, T] \).  
	If \( t \mapsto \varphi(t,s,\omega) \in L^2(\varepsilon \times \mathbb{P}) \) is an \( \mathcal{F} \)-adapted process with  
\begin{eqnarray*}
		\E \left[ \int_t^T \varphi^2(t,s) \, dt \right] < \infty ,
\end{eqnarray*}
and $ F \in L^2(\mathcal{F}_T, \mathbb{P})$. Then we have
	\begin{eqnarray}\label{dd}
		\mathbb{E} \left[ F \int_0^T \varphi(t,s) \, dB(t) \right] 
		= \mathbb{E} \left[ \int_0^T \mathbb{E}[D_t F \mid \mathcal{F}_t] \, \varphi(t,s) \, dt \right].
	\end{eqnarray}
\end{proposition}
\begin{proof}
As observed by Agram and Øksendal~\cite{Agram2015}, for a fixed $s \in [0, T]$, using Proposition \ref{GCO} together with Itô's isometry, we obtain
\begin{eqnarray*}
\mathbb{E} \left[ \left( \int_t^T \varphi(s,t) \, dB(t) \right) F \right] 
		&=& \mathbb{E} \left[ \left( \mathbb{E}[F] + \int_t^T \mathbb{E}[D_t F \mid \mathcal{F}_t] \, dB(t) \right) \left( \int_t^T \varphi(s,t) \, dB(t) \right) \right] \\
		&=& \mathbb{E} \left[ \int_t^T \mathbb{E}[D_t F \mid \mathcal{F}_t] \, \varphi(s,t) \, dt \right].
\end{eqnarray*}
\end{proof}
\begin{theorem}\textbf{(Representation Theorem for advanced BSVIEs)}\label{TRep}\\
	Assume that the generator $\phi: [0, T]^2\times\mathbb{R}\times\mathbb{R}\to \mathbb{R}$ such that $\phi(.,s.,.)$ is $(\mathcal{F}_s)_{s\geq 0}$-adapted. Suppose that there is the couple of process $(p(t), q(t, .))$ solution of  the backward stochastic Volterra integral equation (BSVIE) 
\begin{eqnarray*}
	p(t) = \Phi(t) + \int_t^T \phi(t, s, p(s+\delta), q(t, s))ds 
		- \int_t^T q(t, s)dB(s), \quad t \in [0, T].	
\end{eqnarray*}  
Then, for all $0\leq t<s\leq T$, we have 
\begin{eqnarray}\label{e1}
q(t, s) = \mathbb{E}\left[ D_s p(t) \mid \mathcal{F}_s \right].
\end{eqnarray}
\end{theorem}

\section{Time-advanced BSVIE for adjoint equations }\label{sec:adjoint}

In this section, we focus on deriving the adjoint equation associated with the stochastic Volterra integral equation (SVIE) \eqref{A1} by applying the stochastic maximum principle.
\begin{description}
\item $({\bf A1})$ For all $x,u$, the processes $b(.,s,x,u), \sigma(.,s,x,u)$ and $f(.,s,x,u,p,q)$ are $(\mathcal{F}_{s})_{s\geq 0}$-adapted and twice continuously differentiable $\mathcal{C}^2_{b}$ with respect to $t, x$ and continuously differentiable $\mathcal{C}^1_{b}$ with respect to $u$.
\item $({\bf A2})$ The function $g$ is $\mathcal{F}_{T}$-measurable and   $\mathcal{C}^1_{b}$ class. 
\item $ ({\bf A3})$ The function $t \mapsto q(t,.)$ is of class $C^1$ and
\begin{eqnarray*}
	\E\left[
\int_0^T \int_0^T 
\left(\frac{\partial q(t,s)}{\partial t}\right)^2 ds\,dt
\right] < +\infty.
\end{eqnarray*}
\end{description}

\begin{theorem}{\textbf{(Stochastic Maximum Principle: Adjoint Equation)}}\label{1}\\
Assume $({\bf A1})$-$({\bf A3})$. Then there exists an adapted adjoint process $(p(t), q(.,t))$ satisfying the following advanced backward stochastic Volterra equation  
\begin{eqnarray}\label{p}
p(t) = \frac{\partial g}{\partial x}\big(X(T)\big)+\int_t^T\mu(s,p(s+\delta),q(s,s+\delta))ds- \int_t^T q(t,s)\, dB(s)
\end{eqnarray}
where  
\begin{eqnarray*}
\mu(t,p(t+\delta),q(t,t+\delta))&=&\frac{\partial f}{\partial x}(t+\delta,X^{u}(t),u(t+\delta)) + \frac{\partial b}{\partial x}(t+\delta,t+\delta,X^{u}(t),u(t+\delta))p(t+\delta)\\
&&\frac{\partial \sigma}{\partial x}(t+\delta,t+\delta,X^{u}(t),u(t+\delta))p(t+\delta)\\
&&+\int_{t+\delta}^{T}\left(p(s) \frac{\partial^2 b}{\partial s \partial x}(s,t+\delta,X^{u}(s),u(s))ds\right.\\
&&\left.+q(s,t+\delta) \frac{\partial^2 \sigma}{\partial s \partial x}(s,t+\delta,X^{u}(s),u(s))\right)ds
\end{eqnarray*}
\end{theorem}
\begin{proof}
Le consider the stochastic Volterra integral equation with delay of the form
\begin{eqnarray*}
\left\{
\begin{array}{lll}
 X^{u}(t) & =& \displaystyle \int_{0}^{t}b(t,s,X^{u}(s-\delta),u(s))\,ds +\int_{0}^{t}\sigma(t,s,X^{u}(s-\delta),u(s))\,dB(s)	 
\quad t \in [0,T],\\\\	
X^{u}(t)& =& x_{0}(t), \quad t \in [-\delta, 0],\quad \delta >0,
\end{array}
\right.
\end{eqnarray*}
and recall the optimal problem
\begin{eqnarray*}
\phi(x_{0})=\sup_{u \in \mathcal{U}} J(x_0,u),
\end{eqnarray*}
where the cost functional $J$ is defined by \eqref{cost}.
For sufficiently small $\varepsilon>0$, let us define $u^{\varepsilon}(t) = u^{*}(t) + \varepsilon \beta(t)$, where $u^{*}$ and $\beta$ are respectively optimal and admissible control. Since $\varepsilon>0$ is sufficiently small, $u^{\varepsilon}$ belongs to  $\mathcal{U}$.
Then, according to assumption $({\bf A1})$ and $({\bf A2})$ we have 
\begin{eqnarray}\label{CF}
\frac{dJ(x_0,u^{\varepsilon})}{d\varepsilon}|_{\varepsilon=0}&=&\lim_{\varepsilon \rightarrow 0}\frac{J(x_0,u^{\varepsilon})-J(x_0,u)}{\varepsilon}\nonumber\\
&=&\lim_{\varepsilon\rightarrow 0}\mathbb{E}\left[\int_{0}^{T}\left(\frac{f(t,X^{u^\varepsilon}(t-\delta),u^{\varepsilon}(t))-f(t,X^{u}(t-\delta),u(t))}{\varepsilon}\right)dt\right.\nonumber\\
&&\left.+\frac{g\big(X^{u^{\varepsilon}}(T)\big)-g\big(X^{u}(T)}{\varepsilon}\Big|(X^{u}_s)_{s\in[-\delta,0]}=x_0\right]\nonumber\\
&=& \mathbb{E}\left[\int_{0}^{T}\left(\lim_{\varepsilon\rightarrow 0}\frac{f(t,X^{u^\varepsilon}(t-\delta),u^{\varepsilon}(t))-f(t,X^{u}(t-\delta),u(t))}{\varepsilon}\right)dt\right.\nonumber\\
&&\left.+\lim_{\varepsilon\rightarrow 0}\frac{g\big(X^{u^{\varepsilon}}(T)\big)-g\big(X^{u}(T)}{\varepsilon}\Big|(X^{u}_s)_{s\in[-\delta,0]}=x_0\right]\nonumber\\
&=& \mathbb{E}\left[\int_{0}^{T}\left(\frac{\partial f}{\partial x}(t,X^{u}(t-\delta),u(t))Y(t-\delta)+\frac{\partial f}{\partial u}(t,X^{u}(t-\delta),u(t))\beta(t)\right)dt\right.\nonumber\\
&&\left.+\frac{\partial g}{\partial x}(X^{u}(T))Y(T)\Big|(X^{u}_s)_{s\in[-\delta,0]}=x_0 \right],  
\end{eqnarray}
where we recall that
\begin{eqnarray*}
	Y(t)=\frac{dX^{u^{\varepsilon}}(t)}{d\varepsilon}|_{\varepsilon=0}.
\end{eqnarray*}
On the the other hand, Let us suppose that an adjoint equation of \eqref{A1} is 
the following BSDE.
\begin{equation*}
p(t) = \frac{\partial g}{\partial x}(X^{u}(T))+\int_{t}^{T}h(s)ds - \int_{t}^{T}q(t,s)dB(s), \;\;\;\; t\in[0,T],
\end{equation*}
with it differential form
\begin{eqnarray}\label{PP}
\left\{
\begin{array}{ll}
	dp(t) =& -\left(h(t)+\displaystyle\int_{t}^{T}\frac{\partial q}{\partial t}(t,s) dB(s)\right)dt + q(t,t)dB(t)	\\\\
p(T)=& \frac{\partial g}{\partial x}(X^{u}(T))	
\end{array}
	\right.
\end{eqnarray}
In the sequel of this proof, our goal is to explicitly determine the expression of $\mu$. For this instance, let us apply Itô's formula to  $p(T)Y(T)$. We have
\begin{eqnarray}\label{Ito}
p(T)Y(T)=p(0)Y(0)+\int_0^T p(t)dY(t)+\int_0^TY(t)dp(t)+\int_0^Td\langle p,Y\rangle_t.
\end{eqnarray}
Taking expectation in \eqref{Ito} together with the fact that $p(T)=\frac{\partial g}{\partial x}(X^{u}(T))$, we get
\begin{eqnarray}\label{Eq001}
\E\left[\frac{\partial g}{\partial x}(X^{u}(T))Y(T)\right]=\E \left[ p(0)Y(0)+\int_0^T p(t)dY(t)+\int_0^TY(t)dp(t)+\int_0^Td\langle p,Y\rangle_t\right].
\end{eqnarray}
In view of it definition, we have
\begin{eqnarray*}
Y(t)= \int_{0}^{t}\left(\frac{\partial b}{\partial x}(t,s)Y(s-\delta)
+ \frac{\partial b}{\partial u}(t,s)\beta(s)\right) \,ds 
+\int_{0}^{t}\left(\frac{\partial \sigma}{\partial x}(t,s)Y(s-\delta)
+ \frac{\partial \sigma}{\partial u}(t,s)\beta(s)\right) \,dB(s)
\end{eqnarray*}
and then
\begin{eqnarray*}\label{YY}
dY(t)&=& 
\left[\frac{\partial b}{\partial x}(t,t)Y(t-\delta)+ \frac{\partial b}{\partial u}(t,t)\beta(t)+ \int_{0}^{t}\frac{\partial^{2} b}{\partial t \partial x}(t,s)Y(s-\delta)\,ds + \int_{0}^{t}\frac{\partial^{2} b}{\partial t \partial u}(t,s)\beta(s)\,ds\right.\nonumber \\
 &&\left.+\int_{0}^{t}\frac{\partial^{2}\sigma}{\partial t \partial x}(t,s)Y(s-\delta)\, dB(s)+\int_{0}^{t}\frac{\partial^{2}\sigma}{\partial t \partial u}(t,s)\beta(s)\, dB(s)\right] dt\nonumber\\
 &&+ \left(\frac{\partial \sigma}{\partial x}(t,t)Y(t-\delta)
+ \frac{\partial \sigma}{\partial u}(t,t)\beta(t)\right)dB(t),	
\end{eqnarray*}
where $\phi(t,s)=\phi(t,s,X^{u}(s-\delta),u(s))$ for $\phi=b,\, \sigma$.
Thus, we get
\begin{eqnarray}\label{Eq002}
\E\left(\int_0^T p(t)dY(t)\right)&=&\E\left[\int_0^T\left(\frac{\partial b}{\partial x}(t, t)p(t) Y(t-\delta)+\frac{\partial b}{\partial u}(t, t)p(t) \beta(t)\right)dt\right]\nonumber\\ 
&&+\E\left[\int_0^T \left(\frac{\partial \sigma}{\partial x}(t, t)p(t) Y(t-\delta)+\frac{\partial \sigma}{\partial u}(t, t) p(t)\beta(t)\right)dB(t)\right]\nonumber\\
&&+\E\left[\int_0^Tp(t)\left[\int_0^T\left(\frac{\partial^2 b}{\partial t \partial x}(t, s)Y(s-\delta)+\frac{\partial^2 b}{\partial t \partial u}(t, s) \beta(s)\right){\bf 1 }_{[0,t]}(s) ds\right]dt\right]\nonumber\\
&&+ \E\left[\int_0^T p(t)\left[\int_0^T \left(\frac{\partial^2 \sigma}{\partial t \partial x}(t, s)Y(s-\delta)+\frac{\partial^2 b}{\partial t \partial u}(t, s) \beta(s)\right){\bf 1 }_{[0,t]}dB(s)\right]dt\right]\nonumber\\
&=& \E\left[\int_0^T\left(\frac{\partial b}{\partial x}(t, t)p(t) Y(t-\delta)+\frac{\partial b}{\partial u}(t, t)p(t) \beta(t)\right)dt\right]\nonumber\\ 
&&+I_1+I_2+I_3	
\end{eqnarray}
In view of $({\bf A1})$-$(({\bf A2})$, we have 
\begin{eqnarray}\label{Eq005}
I_1 &=& 0.
\end{eqnarray}
By using Fubini's theorem, we obtain
\begin{eqnarray}\label{Eq003}
 I_2 &=&\mathbb{E}\left[ \int_0^T \left( \int_0^T p(t) \left(\frac{\partial^{2}
	b}{\partial t \partial x}(t, s)Y(s-\delta)+\frac{\partial^{2}
		b}{\partial t \partial u}(t, s)\beta(s)\right){\bf 1 }_{[s,+\infty[}(t)dt\right)ds\right]. 	
\end{eqnarray}
Using again Fubini's Theorem and in virtue of the duality formula  \eqref{dd}, we obtain
\begin{eqnarray*}
I_3 &=&\int_0^T \mathbb{E}\left[ p(t)\int_0^T  \left(\frac{\partial^{2} \sigma}{\partial t \partial x}(t, s)Y(s-\delta)+\frac{\partial^{2} \sigma}{\partial t \partial u}(t, s)\beta(s)\right){\bf 1 }_{[0,t](s)}dB(s)\right] dt\\
	&=&\int_0^T \mathbb{E} \left[ \int_0^T \mathbb{E}\left[ D_s p(t) \mid \mathcal{F}_s \right]\left(\frac{\partial^{2} \sigma}{\partial t \partial x}(t, s)Y(s-\delta)+\frac{\partial^{2} \sigma}{\partial t \partial u}(t, s)\beta(s)\right){\bf 1 }_{[0,t](s)}ds \right] dt\\
	&=&\mathbb{E}\left[\int_0^T \left( \int_0^T \mathbb{E}\left[ D_s p(t) \mid \mathcal{F}_s \right]\left(\frac{\partial^{2} \sigma}{\partial t \partial x}(t, s)Y(s-\delta)+\frac{\partial^{2} \sigma}{\partial t \partial u}(t, s)\beta(s)\right){\bf 1 }_{[s,+\infty]}(t)dt \right) ds\right]\end{eqnarray*}
Next, according to \eqref{e1} in Theorem \ref{TRep}, we have
\begin{eqnarray}\label{Eq004}
	I_3=\mathbb{E}\left[\int_0^T \left( \int_0^T q(t,s)\left(\frac{\partial^{2} \sigma}{\partial t \partial x}(t, s)Y(s-\delta)+\frac{\partial^{2} \sigma}{\partial t \partial u}(t, s)\beta(s)\right){\bf 1 }_{[s,+\infty]}(t)dt \right) ds\right].
\end{eqnarray}
Plugging \eqref{Eq005}-\eqref{Eq004} in \eqref{Eq002} we obtain
\begin{eqnarray}\label{Eq006bis}
\E\left(\int_0^T p(t)dY(t)\right)&=& \E\left[\int_0^T\left(\frac{\partial b}{\partial x}(t, t)p(t) Y(t-\delta)+\frac{\partial b}{\partial u}(t, t)p(t) \beta(t)\right)dt\right.\nonumber\\
&&+\left. \int_0^T \left( \int_0^T p(t) \left(\frac{\partial^{2}
	b}{\partial t \partial x}(t, s)Y(s-\delta)+\frac{\partial^{2}
		b}{\partial t \partial u}(t, s)\beta(s)\right){\bf 1 }_{[s,+\infty[}(t)dt\right)ds\right.\nonumber\\
&&+\left.\int_0^T \left( \int_0^T q(t,s)\left(\frac{\partial^{2} \sigma}{\partial t \partial x}(t, s)Y(s-\delta)+\frac{\partial^{2} \sigma}{\partial t \partial u}(t, s)\beta(s)\right){\bf 1 }_{[s,+\infty]}(t)dt \right) ds\right].\nonumber\\
\end{eqnarray}
On the other hand, it follows from \eqref{PP}, \eqref{YY} that
\begin{eqnarray}\label{PY}
\mathbb{E} \left[ \int_0^T d\langle p, Y \rangle_t \right] = \mathbb{E}\left[ \int_{0}^{T}q(t,t)\left( \frac{\partial \sigma}{\partial x}(t,t)Y(t-\delta) + \frac{\partial \sigma}{\partial x}(t,t)\beta(t)\right)\,dt\right]
\end{eqnarray}
and
\begin{eqnarray}\label{Eq007}
\E\left[\int_{0}^{T}Y(t)dp(t)\right]&=&\E\left[\int_{0}^{T}Y(t)\left(-h(t,X(t),p(t),q(t,t))-\int_{t}^{T}\frac{\partial q}{\partial t}(t,s)dB(s)\right)dt\nonumber\right.\\
&&\left.+\int_0^T Y(t)q(t,t)dB(t)\right] \nonumber\\
&=& -\E\left[\int_{0}^{T}Y(t)h(t,X(t),p(t),q(t,t))dt\right]-\E\left[\int_{0}^{T}Y(t)\left(\int_{t}^{T}\frac{\partial q}{\partial t}(t,s)dB(s)\right)dt\right]\nonumber\\
&&+\E\left[\int_0^T Y(t)q(t,t)dB(t)\right]\nonumber\\
&=&-\E\left[\int_{0}^{T}Y(t)h(t,X(t),p(t),q(t,t))dt\right].
\end{eqnarray}
Indeed, in view of assumption $({\bf A1})$-$({\bf A3})$ and stochastic Fubini Theorem, one can derive easily that 
\begin{eqnarray*}
	\E\left[\int_{0}^{T}Y(t)\left(\int_{t}^{T}\frac{\partial q}{\partial t}(t,s)dB(s)\right)dt\right]&=&\E\left[\int_{0}^{T}\left(\int_{0}^{s}Y(t)\frac{\partial q}{\partial t}(t,s)dt\right)dB(s)\right]\\
	&=& 0
\end{eqnarray*}
and 
\begin{eqnarray*}
	\E\left[\int_0^T Y(t)q(t,t)dB(t)\right]=0.
\end{eqnarray*}
Thus, with \eqref{Eq006bis}, \eqref{PY}, \eqref{Eq007} and \eqref{Eq001} put together, we obtain
\begin{eqnarray*}
\E\left[\frac{\partial g}{\partial x}(X^{u}(T))Y(T)\right]&=& \E\left[\int_0^T\left(\frac{\partial b}{\partial x}(t, t)p(t) Y(t-\delta)+\frac{\partial b}{\partial u}(t, t)p(t) \beta(t)\right)dt\right.\nonumber\\
&&+\left.\int_{0}^{T}q(t,t)\left( \frac{\partial \sigma}{\partial x}(t,t)Y(t-\delta) + \frac{\partial \sigma}{\partial x}(t,t)\beta(t)\right)dt\right.\\
&&+\left. \int_0^T \left( \int_0^T p(t) \left(\frac{\partial^{2}
	b}{\partial t \partial x}(t, s)Y(s-\delta)+\frac{\partial^{2}
		b}{\partial t \partial u}(t, s)\beta(s)\right){\bf 1 }_{[s,+\infty[}(t)dt\right)ds\right.\nonumber\\
&&+\left.\int_0^T \left( \int_0^T q(t,s)\left(\frac{\partial^{2} \sigma}{\partial t \partial x}(t, s)Y(s-\delta)+\frac{\partial^{2} \sigma}{\partial t \partial u}(t, s)\beta(s)\right){\bf 1 }_{[s,+\infty]}(t)dt \right) ds\right.\\
&&-\left. \int_{0}^{T}Y(t)h(t,t,X(t),p(t),q(t,s))dt\right],
\end{eqnarray*} 
where $h(t)=h(t,X(t),p(t),q(t,t))$.\newpage
Finally, we have
\begin{eqnarray*}
&&\frac{dJ(x_0,u^{\varepsilon})}{d\varepsilon}|_{\varepsilon=0}\\ &=&	\mathbb{E}\left[\int_{0}^{T}\left(\frac{\partial f}{\partial x}(t,X^{u}(t-\delta),u(t))Y(t-\delta)+\frac{\partial f}{\partial u}(t,X^{u}(t-\delta),u(t))\beta(t)\right)dt\right.\\
&&+\left.\int_0^T\left(\frac{\partial b}{\partial x}(t, t)p(t) Y(t-\delta)+\frac{\partial b}{\partial u}(t, t)p(t) \beta(t)\right)dt\right.\nonumber\\
&&+\left.\int_{0}^{T}q(t,t)\left( \frac{\partial \sigma}{\partial x}(t,t)Y(t-\delta) + \frac{\partial \sigma}{\partial x}(t,t)\beta(t)\right)dt\right.\\
&&+ \left. \int_0^T \left( \int_0^T p(t) \left(\frac{\partial^{2}
	b}{\partial t \partial x}(t, s)Y(s-\delta)+\frac{\partial^{2}
		b}{\partial t \partial u}(t, s)\beta(s)\right){\bf 1 }_{[s,+\infty[}(t)dt\right)ds \right.\nonumber\\
&&+\left.\int_0^T \left( \int_0^T q(t,s)\left(\frac{\partial^{2} \sigma}{\partial t \partial x}(t, s)Y(s-\delta)+\frac{\partial^{2} \sigma}{\partial t \partial u}(t, s)\beta(s)\right){\bf 1 }_{[s,+\infty]}(t)dt \right) ds+\int_{0}^{T}Y(t)dp(t)\right]\\
&=&\mathbb{E}\left[\int_{0}^{T-\delta}\left(\frac{\partial f}{\partial x}(t+\delta,X^{u}(t),u(t+\delta))Y(t)+\frac{\partial f}{\partial u}(t+\delta,X^{u}(t),u(t+\delta))\beta(t+\delta)\right)dt\right.\\
&&\left.\int_0^{T-\delta}\left(\frac{\partial b}{\partial x}(t+\delta, t+\delta)\,p(t+\delta)\,Y(t)+\frac{\partial b}{\partial u}(t+\delta, t+\delta)\,p(t+\delta)\, \beta(t+\delta)\right)dt\right.\nonumber\\
&&+\left.\int_{0}^{T-\delta}q(t+\delta,t+\delta)\left( \frac{\partial \sigma}{\partial x}(t+\delta,t+\delta)\,Y(t) + \frac{\partial \sigma}{\partial x}(t+\delta,t+\delta)\,\beta(t+\delta)\right)dt\right.\\
&&+\left.\int_0^T \left( \int_0^T p(r) \left(\frac{\partial^{2}
	b}{\partial r \partial x}(r, s+\delta)Y(s)+\frac{\partial^{2}
		b}{\partial r \partial u}(r, s+\delta)\beta(s+\delta)\right){\bf 1 }_{[s+\delta,+\infty[}(r)dr\right)ds \right.\nonumber\\
&&+\left.\int_0^T \left( \int_0^T q(r,s+\delta)\left(\frac{\partial^{2} \sigma}{\partial r \partial x}(r, s+\delta)Y(s)+\frac{\partial^{2} \sigma}{\partial r \partial u}(r, s+\delta)\beta(s+\delta)\right){\bf 1 }_{[s+\delta,+\infty[}(r)dr \right) ds\right.\\
&&\left.-\int_{0}^{T}Y(t)h(t)dt\right].
\end{eqnarray*}
Since adjoint process in order to eliminate the terms involving $Y$ and to express the variation solely in terms of the control perturbation, we need for all $t\in [0,T-\delta]$,
\begin{eqnarray*}
\frac{\partial f}{\partial x}(t+\delta,X^{u}(t),u(t+\delta)) 
			+ \frac{\partial b}{\partial x}(t+\delta,t+\delta)p(t+\delta)
			+ \frac{\partial \sigma}{\partial x}(t+\delta,t+\delta)q(t+\delta,t+\delta)\\
			+\int_{t+\delta}^{T}\left(p(s) \frac{\partial^2 b}{\partial s \partial x}(s,t+\delta)+q(s,t+\delta) \frac{\partial^2 \sigma}{\partial s \partial x}(s,t+\delta)\right)ds-h(t)	&=& 0 
\end{eqnarray*}
so that 
\begin{eqnarray*}
h(t)&=&\left[\frac{\partial f}{\partial x}(t+\delta,X^{u}(t),u(t+\delta)) + \frac{\partial b}{\partial x}(t+\delta,t+\delta)p(t+\delta)+ \frac{\partial \sigma}{\partial x}(t+\delta,t+\delta)q(t+\delta,t+\delta)\right.\\
&&\left.+\int_{t+\delta}^{T}\left(p(s) \frac{\partial^2 b}{\partial s \partial x}(s,t+\delta)+q(s,t+\delta) \frac{\partial^2 \sigma}{\partial s \partial x}(s,t+\delta)\right)ds\right]{\bf 1}_{[0,T-\delta]}(t)\\
&=& \frac{\partial \mathcal{H}}{\partial x}(t),
\end{eqnarray*}
where the function $\mathcal{H}$ is the Hamiltonian functional associated to our control problem \eqref{CP} defined by
\begin{eqnarray*}
\mathcal{H}(t,x,u,p,q)&=&\left[f(t+\delta,x,u)+b(t+\delta,t+\delta,x,u)p(t+\delta)+\sigma(t+\delta,t+\delta,x,u)q(t+\delta,t+\delta\right.)\\
	&&\left.+\int_{t+\delta}^T\left(p(s)\frac{\partial b}{\partial s}(s,t+\delta)+q(s,t+\delta)\frac{\partial \sigma}{\partial s}(s,t+\delta)\right)ds\right]{\bf 1}_{[0,T-\delta]}.
 \end{eqnarray*}
In virtue of the above condition, we have
\begin{eqnarray*}
&&\frac{dJ(x_0,u^{\varepsilon})}{d\varepsilon}|_{\varepsilon=0}\\
&=&\mathbb{E}\left[\int_{0}^{T-\delta}\left(\frac{\partial f}{\partial u}(t+\delta,X^{u}(t),u(t+\delta))\beta(t+\delta)+\frac{\partial b}{\partial u}(t+\delta, t+\delta)\,p(t+\delta)\beta(t+\delta)\right.\right.\\
&&\left.\left.+\frac{\partial \sigma}{\partial u}(t+\delta, t+\delta)\,q(t+\delta,t+\delta)\beta(t+\delta)\right.\right.\\
&&\left.\left.+\int_{s+\delta}^T\left(\frac{\partial^{2}b}{\partial r \partial u}(r, s+\delta)\beta(s+\delta)p(r)+\frac{\partial^{2}\sigma}{\partial r \partial u}(r, s+\delta)\beta(t+\delta)q(r,t+\delta)\right)dr\right)dt\right]\\
&=& \mathbb{E}\left[\int_{0}^{T}\frac{\partial \mathcal{H}}{\partial u}(t)dt\right]
\end{eqnarray*}
To end this proof, let us observe that equation \eqref{p} is an anticipated Volterra-type BSDE. This type of BSDE was first studied by Jiaqiang Wen and Yufeng Shi \cite{wen2017a}. Among other things, they established a existence and uniqueness result under a global Lipschitz condition. Therefore, since functions $b,\sigma$ and $f$ satisfy assumptions $(A1)$-$(A3)$, the fonction $\mu$ is Lipchtz with respecct $p$ and $q$ so that BSDE \eqref{p} admit a unique solution. 
\end{proof}

\section{Sufficient Maximum Principle}\label{sec:sufficient}
In this section, we establish a stochastic maximum principle for delayed stochastic Volterra integral systems under partial information. Let $(\Omega,\mathcal{F},(\mathcal{F}_t)_{t \geq 0},\mathbb{P})$ be a filtered probability space satisfying the usual conditions. The information available to the controller is modeled by $(\mathcal{G}_t)_{t \geq 0}$ a sub-filtration of  $(\mathcal{F}_t)_{t \geq 0}$. Let $U \subset \mathbb{R}$ be a nonempty convex set. We define the set of admissible controls, denoted by $\mathcal{A}_{\mathcal{G}}$, as the collection of all $U$-valued processes $u = (u_t)_{t \geq 0}$, $(\mathcal{G}_t)_{t \geq 0}$-adapted and càdlàg. For notational simplicity, we will write $X^{u}(t) = X(t)$.
\begin{theorem}
Let $\hat{u} \in \mathcal{A}_{\mathcal{G}}$ and denote by $\hat{X}(t)$ the corresponding state process. Let $\big(\hat{p}(t), \hat{q}(t,s)\big)$ be the associated adjoint processes, which are solutions of equations \eqref{p}.
Assume that the following conditions hold:
\begin{itemize}
\item[(i)] The function $x \mapsto g(x)$ and the Hamiltonian $(x, u) \mapsto \mathcal{H}(t, x, u, \hat{p}(t), \hat{q}(.,t))$ are concave for each $t \in [0,T]$, almost surely.
\item[(ii)] Assumption $({\bf A3})$ holds for all $u \in \mathcal{A}_{\mathcal{G}}$.
\item[(iii)] (Maximum condition) For all $t \in [0,T]$, it holds that
\begin{eqnarray}\label{Ep002}
	\mathbb{E} \left[ \mathcal{H}(t, \hat{X}(t), \hat{u}(t), \hat{p}(t), \hat{q}(\cdot,t))|\mathcal{G}_t\right]
= \max_{v \in U} \mathbb{E} \left[\mathcal{H}(t, \hat{X}(t), v(t), \hat{p}(t), \hat{q}(\cdot,t)) \,\big|\mathcal{G}_t\right].
\end{eqnarray}
\end{itemize}
Then, the control $\hat{u}$ is optimal for the stochastic control problem \eqref{CP}.
\end{theorem}
\begin{proof}
Consider the state process $\hat{X}(t)$ associated with the admissible control $\hat{u} \in \mathcal{A}_{\mathcal{G}}$, defined by the delayed stochastic Volterra integral equation
\begin{eqnarray*}
\left\{
\begin{array}{lll}
\hat{X}(t) & = & \displaystyle \int_{0}^{t} b(t,s,\hat{X}(s-\delta),\hat{u}(s))ds+ \int_{0}^{t} \sigma\big(t,s,\hat{X}(s-\delta),\hat{u}(s)\big)dB(s), \quad t \in [0,T], \\\\
\hat{X}(t) & = & x_{0}(t), \quad t \in [-\delta,0],
\end{array}
\right.
\end{eqnarray*}
where $\delta > 0$ is a given delay.
Let $u \in\mathcal{A}_{\mathcal{G}}$ be any admissible control. Our objective is to prove that
\begin{eqnarray*}
J(x_{0},u) \leq J(x_{0},\hat{u})	,
\end{eqnarray*}
which implies that $\hat{u}$ is an optimal control for stochastic problem \eqref{CP}.
According to \eqref{cost}, we have 
\begin{eqnarray*}
J(x_0,u) - J(x_0,\hat{u})&=& \mathbb{E} \left[ \int_0^T \Big( f\big(t, X(t-\delta), u(t)\big) - f\big(t, \hat{X}(t-\delta), \hat{u}(t)\big) \Big), dt \right]\\
&& \mathbb{E} \left[ g\big(X(T)\big) - g\big(\hat{X}(T)\big) \right]\\
&=& I_1 + I_2.
\end{eqnarray*}

In view of the definition of $\mathcal{H}$ and the assumptions related to them, we have
\begin{eqnarray}\label{I1}
I_1 &=& \mathbb{E} \left[\int_0^T \left\{ \mathcal{H}(t, X(t-\delta), u(t))  - \mathcal{H}(t,\hat{X}(t-\delta), \hat{u}(t))\right.\right.\nonumber\\
&&\left.\left.- \hat{p}(t)(b(t,t, X(t-\delta),  u(t)) - b(t,t, \hat{X}(t-\delta), \hat{u}(t))\right.\right.\nonumber\\
&&\left.\left. - \hat{q}(t,t)(\sigma(t,t, X(t-\delta), u(t)) - \sigma(t,t, \hat{X}(t-\delta), \hat{u}(t)))\right.\right.\nonumber\\
&& \left.\left.- \int_{t}^{T}\hat{p}(s)\left( \frac{\partial b}{\partial s}(s,t,X(s-\delta),u(s)) - \frac{\partial b}{\partial s}(s,t,\hat{X}(s-\delta),\hat{u}(s))\right)ds\right.\right.\nonumber \\
&& \left.\left.- \int_{t}^{T}\hat{q}(s,t)\left( \frac{\partial \sigma}{\partial s}(s,t,X(s-\delta),u(s)) - \frac{\partial \sigma}{\partial s}(s,t,\hat{X}(s-\delta),\hat{u}(s))\right)ds\right\} dt \right]	\nonumber\\
&\leq & \mathbb{E}\left[\int_0^T \left\{\frac{\partial \hat{\mathcal{H}}}{\partial x}(t) (X(t-\delta) - \hat{X}(t-\delta))+\frac{\partial \hat{\mathcal{H}}}{\partial u}(t) (u(t) - \hat{u}(t))- \hat{p}(t)(b(t,t) - \hat{b}(t,t))\nonumber\right.\right. \nonumber\\
&&\left.\left.- \hat{q}(t,t)( \sigma(t,t) - \hat{\sigma}(t,t))- \int_{t}^{T}\hat{p}(s)(\frac{\partial b}{\partial s}(s,t)- \frac{\partial \hat{b}}{\partial s}(s,t)) ds \nonumber\right.\right.\nonumber \\
&&\left.\left. - \int_{t}^{T}\hat{q}(s,t)\left( \frac{\partial \sigma}{\partial s}(s,t)- \frac{\partial \hat{\sigma}}{\partial s}(s,t)\right)ds\right\}dt\right], 
\end{eqnarray}
Using the concavity of the function $g$ together with the terminal condition associated with the BSVIE, we deduce that
\begin{eqnarray*}
I_2 &\leq & \mathbb{E} \left[ \frac{\partial g}{\partial x}\big(\hat{X}(T)\big) \cdot \big(\hat{X}(T) - \hat{X}(T)\big) \right], \\
	&\leq &\mathbb{E} \left[ \hat{p}(T) \cdot \big(X(T) - \hat{X}(T)\big) \right].	
\end{eqnarray*}

Next, applying Itô's formula to $\hat{p}(T)(X(T)-\hat{X}(T))$ and taking the expectation, we have
\begin{eqnarray}\label{Eq0005}
	\E\left[ \frac{\partial g}{\partial x}(\hat{X}(T))\left(X(T)-\hat{X}(T)\right)\right]  &=&\E\left[ \hat{p}(0)\left( X(0)-\hat{X}(0)\right)+\int_0^T \hat{p}(t)\,d\left( X(t)-\hat{X}(t)\right)\right.\nonumber \\
	&&\left. +\int_0^T\left( X(t)-\hat{X}(t)\right)\,d\hat{p}(t)
	+\int_0^T\,d\langle \hat{p},\left( X-\hat{X}\right)\rangle_t\right]\nonumber\\
&=&	\E\left[\int_0^T \hat{p}(t)\,d\left( X(t)-\hat{X}(t)\right)+\int_0^T\left( X(t)-\hat{X}(t)\right)\,d\hat{p}(t)\right.\nonumber\\
	&&\left.+\int_0^T\,d\langle \hat{p},\left( X-\hat{X}\right)\rangle_t\right]\nonumber\\
	&=&J_1+J_2+J_3
\end{eqnarray}
We have
\begin{eqnarray*}
	J_1&=&\E\left[\int_0^T\hat{p}(t)\left(b(t,t)-\hat{b}(t,t)\right)dt\right]\nonumber\\ 
	&&+\E\left[\int_0^T\hat{p}(t) \left(\sigma(t,t)-\hat{\sigma}(t,t)\right)dB(t)\right]\nonumber\\
	&&+\E\left[\int_0^T\hat{p}(t)\left( \int_0^t\left( \frac{\partial b}{\partial t }(t, s)-\frac{\partial \hat{b}}{\partial t }(t, s) \right) ds\right) dt\right]\nonumber\\
	&&+ \E\left[\int_0^T \hat{p}(t)\left(  \int_0^t\left(  \frac{\partial \sigma}{\partial t }(t, s)- \frac{\partial \hat{\sigma}}{\partial t }(t, s)\right) dB(s)\right) \right]dt\nonumber\\
	&=& \E\left[\int_0^T\hat{p}(t)\left(b(t,t)-\hat{b}(t,t)\right)dt\right]+r_1+r_2.	
\end{eqnarray*}
Using Fubini's theorem, the duality formula \eqref{dd} and equality \eqref{e1} in Theorem \ref{TRep}, we obtain respectively

\begin{eqnarray}\label{Eq003}
	r_1 &=&\mathbb{E}\left[ \int_0^T \left( \int_s^T \hat{p}(t) \left(\frac{\partial
		b}{\partial t }(t, s) -\frac{\partial
		\hat{b}}{\partial t }(t, s)\right)dt\right)ds\right]\nonumber\\
	&=& \mathbb{E}\left[ \int_0^T \left( \int_t^T \hat{p}(s) \left(\frac{\partial
		b}{\partial s }(s,t)-\frac{\partial
		\hat{b}}{\partial s }(s, t)\right)ds\right)dt\right]\nonumber. 	
\end{eqnarray}
and
\begin{eqnarray*}
	r_2 &=&\int_0^T \mathbb{E}\left[ \hat{p}(t)\int_0^t  \left(\frac{\partial \sigma}{\partial t }(t, s)-\frac{\partial \hat{\sigma}}{\partial t }(t, s)\right)dB(s)\right] dt\\
	&=&\int_0^T \mathbb{E} \left[ \int_0^t \mathbb{E}\left[ D_s \hat{p}(t) \mid \mathcal{F}_s \right]\left(\frac{\partial \sigma}{\partial t}(t, s)-\frac{\partial \hat{\sigma}}{\partial t}(t, s)\right)ds \right] dt\\
	&=& \mathbb{E}\left[\int_0^T \left( \int_t^T \mathbb{E}\left[ D_t p(s) \mid \mathcal{F}_t \right]\left(\frac{\partial \sigma}{\partial s}(s, t)-\frac{\partial \hat{\sigma}}{\partial s }(s, t)\right)ds \right) dt\right]\nonumber\\
&=&	\mathbb{E}\left[\int_0^T \left( \int_t^T \hat{q}(s,t)\left( \frac{\partial \sigma}{\partial s }(s, t)-\frac{\partial \hat{\sigma}}{\partial s }(s, t)\right) \,ds \right) dt\right],
\end{eqnarray*}
which imply thay
\begin{eqnarray}\label{Eq0021}
J_1&=&\E\left[\int_0^T\left(\hat{p}(t)\left(b(t,t)-\hat{b}(t,t)\right)+\int_t^T \hat{p}(s)\left(\frac{\partial b}{\partial s }(s,t)-\frac{\partial \hat{b}}{\partial s }(s, t)\right)ds\right.\right.\nonumber\\
&&+\left.\left.\int_t^T \hat{q}(s,t)\left( \frac{\partial \sigma}{\partial s }(s, t)-\frac{\partial \hat{\sigma}}{\partial s }(s, t)\right)ds \right)dt\right].
\end{eqnarray}

On the other hand, we have
\begin{eqnarray}\label{PY1}
	J_3= \mathbb{E}\left[ \int_{0}^{T}\hat{q}(t,t)\left(\sigma(t,t)-\hat{\sigma}(t,t)\right)\,dt\right]
\end{eqnarray}
and
\begin{eqnarray}\label{Eq0071}
	J_2&=&\E\left[\int_{0}^{T}\left( X(t)-\hat{X}(t)\right) \left(-\frac{\partial \hat{\mathcal{H}}}{\partial x}(t)\mathbf{1}_{[0,T-\delta]}-\int_{t}^{T}\frac{\partial \hat{q}}{\partial t}(t,s)dB(s)\right)dt\right.\nonumber \\
	&&\left.+\int_0^T \left( X(t)-\hat{X}(t)\right) \hat{q}(t,t)\,dB(t)\right] \nonumber\\
	&=& -\E\left[\int_{0}^{T}\left( X(t)-\hat{X}(t)\right) \frac{\partial \hat{\mathcal{H}}}{\partial x}(t)\mathbf{1}_{[0,T-\delta]}dt\right]\nonumber \\
	&&-\E\left[\int_{0}^{T}\left( X(t)-\hat{X}(t)\right) \left(\int_{t}^{T}\frac{\partial \hat{q}}{\partial t}(t,s)dB(s)\right)dt\right]\nonumber\\
	&&+\E\left[\int_0^T \left( X(t)-\hat{X}(t)\right) \hat{q}(t,t)\,dB(t)\right]\nonumber\\
	&=&-\E\left[\int_{0}^{T}\left( X(t)-\hat{X}(t)\right) \frac{\partial \hat{\mathcal{H}}}{\partial x}(t)\mathbf{1}_{[0,T-\delta]}\,dt\right]
\end{eqnarray}
Indeed, in view of assumption $({\bf A1})$-$({\bf A3})$ and stochastic Fubini Theorem, one can derive easily that 
\begin{eqnarray*}
	\E\left[\int_{0}^{T}\left( X(t)-\hat{X}(t)\right) \left(\int_{t}^{T}\frac{\partial \hat{q}}{\partial t}(t,s)dB(s)\right)dt\right]&=&\E\left[\int_{0}^{T}\left(\int_{0}^{s}\left(X(t)-\hat{X}(t)\right) \frac{\partial \hat{q}}{\partial t}(t,s)dt\right)dB(s)\right]\\
	&=& 0
\end{eqnarray*}
and 
\begin{eqnarray*}
\E\left[\int_0^T \left( X(t)-\hat{X}(t)\right) \hat{q}(t,t)dB(t)\right]=0.
\end{eqnarray*}
Next, it follows from \eqref{Eq0021}, \eqref{PY1} and \eqref{Eq0071} that
\begin{eqnarray}\label{Ep008}
I_2 &\leq & \E \left[\int_0^T \left(\hat{p}(t)\,\left(b(t,t)-\hat{b}(t,t)\right)
	+ \hat{q}(t,t)\,\left( \sigma(t,t)-\hat{\sigma}(t,t)\right)+ \int_t^T \hat{p}(s) \left( \frac{\partial b}{\partial s}(s,t)-\frac{\partial \hat{b}}{\partial s}(s,t)\right)ds\right.\right.  \nonumber\\
	&&\left.\left.+ \int_t^T \hat{q}(s,t)\left( \frac{\partial \sigma}{\partial s}(s,t)- \frac{\partial \hat{\sigma}}{\partial s}(s,t)\right) \, ds -\frac{\partial \hat{\mathcal{H}}}{\partial x}(t)\mathbf{1}_{[0,T-\delta]}(t) \left(X(t)-\hat{X}(t)\right) 
	\right)dt\right].
\end{eqnarray}
Finally, in virtue of inequalities \eqref{I1} and \eqref{Ep008} we derive
\begin{eqnarray*}
J(x_0,u)-J(x_0,\hat{u}) &\leq &\mathbb{E}\left[\int_0^T\left\{\frac{\partial \hat{\mathcal{H}}}{\partial x}(t) \left(X(t) - \hat{X}(t)\right)+  \frac{\partial \hat{\mathcal{H}}}{\partial u}(t)(u(t) - \hat{u}(t))\right.\right.\\
	&&\left.\left.- \hat{p}(t)\,\left( b(t,t) - \hat{b}(t,t)\right) - \hat{q}(t,t)\,\left( \sigma(t,t) - \hat{\sigma}(t,t)\right)\right.\right.\\
	&&-\left.\left. \int_{t}^{T}\hat{p}(s)\left( \frac{\partial b}{\partial s}(s,t)- \frac{\partial \hat{b}}{\partial s}(s,t)\right) ds  - \int_{t}^{T}\hat{q}(s,t)\left( \frac{\partial \sigma}{\partial s}(s,t)- \frac{\partial \hat{\sigma}}{\partial s}(s,t)\right)ds\right.\right.\\ 
	&& \left.\left.+\hat{p}\,(t)\left(b(t,t)-\hat{b}(t,t)\right)
	+\hat{q}(t,t)\,\left( \sigma(t,t)-\hat{\sigma}(t,t)\right)\right.\right.\\
	&&+\left.\left. \int_t^T \hat{p}(s) \left( \frac{\partial b}{\partial s}(s,t)-\frac{\partial \hat{b}}{\partial s}(s,t)\right) \, ds + \int_t^T \hat{q}(s,t)\left( \frac{\partial \sigma}{\partial s}(s,t)- \frac{\partial \hat{\sigma}}{\partial s}(s,t)\right)ds\right.\right.\\
	&& \left.\left.- \frac{\partial \hat{\mathcal{H}}}{\partial x}(t)\mathbf{1}_{[0,T-\delta]}(t) \left(X(t)-\hat{X}(t)\right) 
	\right\}dt \right]\\
&\leq &\int_0^T \mathbb{E} \left[\frac{\partial \hat{\mathcal{H}}}{\partial u}(t)(u(t) - \hat{u}(t)) \right] dt.
	\end{eqnarray*}
Since $u$ and $\hat{u}$ are $(\mathcal{G}_t)_{t\geq 0}$-adapted and $\hat{u}$ maximizes the conditional Hamiltonian so that 
\begin{eqnarray*}
\mathbb{E}\left[\frac{\partial \hat{\mathcal{H}}}{\partial u}(t)|\mathcal{G}_t\right]=0,
\end{eqnarray*}
we have
\begin{eqnarray*}
J(x_{0},u) - J(x_{0},\hat{u}) &\leq &  \int_0^T \mathbb{E} \left[\E\left(\frac{\partial \hat{\mathcal{H}}}{\partial u}(t)|\mathcal{G}_t\right)(u(t)-\hat{u}(t))\right]dt \\
	&\leq& 0.
\end{eqnarray*}
This proves that $\hat{u}$ is an optimal control.
\end{proof}

\section{A necessary maximum principle}\label{sec:necessary}
One limitation of the sufficient maximum principle presented in Section 4 is the concavity assumption, which is not always satisfied in practical applications. In this section, we establish a complementary result in the reverse direction. More precisely, we show that being a directional critical point of the performance functional $J(x_{0},u)$ is equivalent to being a critical point of the conditional Hamiltonian. For this, we introduce the following assumptions
\begin{description}
\item[(A4)] For all $u \in \mathcal{A}_{\mathcal{G}}$ and all bounded $\beta \in \mathcal{A}_{\mathcal{G}}$, there exists $\varepsilon>0$ such that $u+s\beta\in\mathcal{A}_{\mathcal{G}}$ for all $s\in[-\varepsilon,\varepsilon]$.
\item[(A5)] For all  $t_0, h$ such that $t_0\leq t_0+h\leq T$  and all bounded $(\mathcal{G}_{t_0})$-measurable random variables $\alpha$, the control process $(\beta(t))_{t\geq 0}$ defined by $\beta(t)=\alpha {\bf 1}_{[t_0, t_0+h]}(t)$ belongs to $\mathcal{A}_{\mathcal{G}}$.
\item[(A6)] For every bounded $\beta\in\mathcal{A}_{\varepsilon}$, the directional derivative process 
\begin{eqnarray}\label{Ep02}
Y(t)=\frac{d}{ds}X^{u+s\beta}(t)|_{s=0},
\end{eqnarray} 
exists and belongs to $L^{2}(\lambda \times \mathbb{P})$, where $\lambda$ denotes the Lebesgue measure on $\R$.
\end{description}
We are now in a position to formulate the following result.
\begin{theorem}[Necessary Maximum Principle]
Assume $({\bf A4)}$-$({\bf A6})$ hold. Let $\hat{u} \in \mathcal{A}_{\mathcal{G}}$, and denote by $\hat{X}(t)$  the corresponding state process solving \eqref{ee1}. Let $(\hat{p}(t), \hat{q}(\cdot,t))$ be the adjoint processes associated with $\hat{u}$, solution of \eqref{p}. Then, the following statements are equivalent.
\begin{itemize}
\item[(i)] For every bounded $\beta\in \mathcal{A}_{\mathcal{G}}$, we have
\begin{eqnarray*}
\lim_{s \to 0}\frac{J(x_{0},\hat{u}+s\beta)-J(x_{0},\hat{u})}{ds} =\frac{dJ}{ds}(x_0,\hat{u}+s\beta)|_{s=0}= 0
\end{eqnarray*}
\item[(ii)] For all $t \in [0,T]$, we have
\begin{eqnarray*}
\mathbb{E} \left[ \left. \frac{\partial \mathcal{H}}{\partial u} \left(t,\hat{X}(t), \hat{u}(t), \hat{p}(t), \hat{q}(\cdot,t) \right) \, \right| \, \mathcal{G}_t \right]= 0 \quad \text{p.s.}		
\end{eqnarray*}
\end{itemize}
\end{theorem}
\begin{proof}
Following the same argument of the proof of Theorem \ref{1}, we derive
\begin{eqnarray*}
\frac{dJ}{ds}(x_0,u+s\beta)|_{s=0}&=&	\E \left[ 
	\int_{0}^{T} Y(t-\delta)\frac{\partial \hat{\mathcal{H}}}{\partial x}(t)dt -\int_{0}^{T} Y(t)\frac{\partial \hat{\mathcal{H}}}{\partial x}(t+\delta)\mathbf{1}_{[0,T-\delta]}(t) \, dt + \int_{0}^{T} \frac{\partial \hat{\mathcal{H}}}{\partial u}(t) \beta(t)  \,dt 
	\right]\\
	&=& \E\left[\int_{0}^{T}\left(Y(t)\frac{\partial \hat{\mathcal{H}}}{\partial x}(t+\delta){\bf1}_{[0,T-\delta]}(t) - Y(t)\frac{\partial \hat{\mathcal{H}}}{\partial x}(t+\delta)\mathbf{1}_{[0,T-\delta]}(t) \right)dt\right.\\
	&&\left.+ \int_{0}^{T} \frac{\partial \hat{\mathcal{H}}}{\partial u}(t) \beta(t) dt\right]\\ 
	&=& \E\left[\int_{0}^{T} \frac{\partial \hat{\mathcal{H}}}{\partial u}(t) \beta(t) dt\right],
\end{eqnarray*}
where $\hat{\mathcal{H}}(t)=\mathcal{H}\left(t,\hat{X}(t), \hat{u}(t), \hat{p}(t), \hat{q}(\cdot,t) \right)$.

Next, in view of $({\bf A5})$, we get
\begin{eqnarray}\label{Ep07}
\frac{dJ}{ds}(x_0,u+s\beta)|_{s=0} 
&=& \mathbb{E}\left[ \int_0^T \frac{\partial \hat{\mathcal{H}}}{\partial u}(s)\, \beta(s)\, ds \right] \nonumber\\
&=& \mathbb{E} \left[ \alpha \int_t^{t+h} \frac{\partial \hat{\mathcal{H}}}{\partial u}(s)\, ds\right]
\end{eqnarray}
According \eqref{Ep07}, if we suppose that
\begin{eqnarray}\label{Ep08}
\frac{dJ(x_0,u+s\beta)}{ds}|_{s=0} = 0,
\end{eqnarray}
hence we have 
\begin{eqnarray*}
\mathbb{E} \left[ \alpha \int_t^{t+h} \frac{\partial \hat{\mathcal{H}}}{\partial u}(s)ds\right]=0
\end{eqnarray*}
Differentiating the the above equality at $h=0$, we get
\begin{eqnarray*}
\left.\frac{d}{dh}\mathbb{E}\left[\alpha \int_{t}^{t+h}\frac{\partial \hat{\mathcal{H}}}{\partial u}(s)ds\right]\right|_{h=0}&=&
\mathbb{E}\left[\alpha \frac{\partial \hat{\mathcal{H}}}{\partial u}(t)\right]\\
&=&0,
\end{eqnarray*}
which implies, since this holds for all bounded $\mathcal{G}_t$-measurable random variable $\alpha$, that
\begin{eqnarray*}
\mathbb{E}\left[\frac{\partial \hat{\mathcal{H}}}{\partial u}(t)|\mathcal{G}_t \right] = 0.
\end{eqnarray*}
Conversely, if we assume that $(ii)$ holds, then we obtain $(i)$ by reversing the argument we used to obtain \eqref{Ep07}.
\end{proof}

\section{Optimal consumption of a delayed Volterra type cash flow}
Let $X^u(t)$ denote a given cash flow, modeled by the following stochastic Volterra equation with delay:
\begin{eqnarray}
	X^{u}(t) &=& x_0 + \int_0^t \big[ b_0(t,s) X^{u}(s-\delta) - u(s) \big] \, ds + \int_0^t \sigma_0(s) X^{u}(s-\delta) \, dB(s) 
	, \quad t \ge 0.
\end{eqnarray}
Equivalently, the dynamics can be written in differential form as
\begin{eqnarray}
\left\{
\begin{array}{l}
dX^{u}(t) = \big[ b_0(t,t) X^{u}(t-\delta) - u(t) \big]dt+ \sigma_0(t) X^{u}(t-\delta) \, dB(t) + \left(\displaystyle \int_0^t \frac{\partial b_0}{\partial t}(t,s) X^{u}(s-\delta) \, ds \right) dt, \quad t \ge 0,	\\\\
X(0) = x_0.
\end{array}
\right.
\end{eqnarray}
We note that the dynamics of $X^{u}(t)$ involve a memory effect, captured by the integral term with respect to $ds$. We assume that $b_0(t,s)$ and $\sigma_0(s)$ are given deterministic functions with values in $\R$, and that $b_0(t,s)$ is continuously differentiable with respect to $t$ for each fixed  $s$. For simplicity, we further assume that these functions are bounded.

Our objective is to solve the following maximization problem which consist to find $\hat{u} \in \mathcal{A}_{\mathbb{G}}$ such that
\begin{eqnarray}\label{OP}
	J(x_{0},\hat{u})= \sup_{u\in \mathcal{G}_t} J(x_{0},u),
\end{eqnarray}
where the performance functional $J$ is defined by
\begin{equation}\label{j}
J(x_{0},u) = \mathbb{E}\left[\theta\, X(T)+\int_0^T \log(u(t))dt\right],
\end{equation}
where $\theta$ is a given $\mathcal{F}_T$-measurable random variable.

In light of the results established in Sections 3, 4, and 5, we obtain the following result, which provides an explicit characterization	 the optimal control $\hat{u}$ solving problem \eqref{OP}. We first give the following statement. Let us define, for $0 \leq t\leq  T-\delta$ 
\begin{eqnarray*}
\varphi_1(t) = \int_{t+\delta}^{T} b_0(t+\delta,s_1)ds_1,\;\;\;\;\;\
\varphi_2(t) = \int_{t+\delta}^T\int_{s_1+\delta}^T b_0(t+\delta,s_1)b_0(s_1+\delta,s_2) ds_1ds_2,
\end{eqnarray*}
and recursively
\begin{eqnarray*}
\varphi_{n-1}(t) = \int_{t+\delta}^T\int_{s_1+\delta}^T\cdots\int_{s_{n-1}+\delta}^T \prod_{k=0}^{n-1} b_0(s_k+\delta,s_{k+1})ds_1d\cdots s_n,\;\;\;\;\;\; n\geq 3.
\end{eqnarray*}
We have 
\begin{remark}
Since we assume the function $b_0$ bounded, there exists a constant $C>0$ such that $|b_0(t+\delta,s)|\leq C$. Moreover, by induction method we prove that
for all $n \geq 1$,
\begin{eqnarray*}
|\varphi_n(t)| \leq \frac{C^n T^n}{n!}.
\end{eqnarray*}
Hence, for all $t$
\begin{eqnarray*}
\sum_{n=1}^{+\infty}|\varphi_n(t)|<+\infty.
\end{eqnarray*}
In the sequel let us set
\begin{eqnarray}\label{Res}
\Psi(t)=\sum_{n=1}^{+\infty}\varphi_n(t).
\end{eqnarray}
\end{remark}
	
\begin{theorem}
Let $\hat{u}$ be the optimal solution to the consumption problem \eqref{OP}. Then, we derive
\begin{eqnarray*}
\hat{u}(t)&=& \left(\mathbb{E}\left[ \mathbb{E}^{\mathbb{Q}}\left(\Psi(t)\theta\middle|\mathcal{F}_t\right)|\mathcal{G}_t\right]\right)^{-1}\\
&=& \left(\mathbb{E}\left[\frac{\mathbb{E}\left(\Psi(t)M(T)\theta\middle|\mathcal{F}_t\right)}{M(t)}\middle|\mathcal{G}_t\right]\right)^{-1},
\end{eqnarray*}
where $\mathbb{Q}$ is a probability measure defined by
\begin{eqnarray}\label{Gir}
	 d\mathbb{Q}=M(T)d\P
\end{eqnarray}
 with
\begin{eqnarray*}
	M(t)=\exp\left(\int_0^t\sigma_0(s)dB(s)-\frac{1}{2}\int_0^t\sigma_0^2(s)ds\right).
\end{eqnarray*}
\end{theorem}
If, in addition, $\mathcal{F}=\mathcal{G}$, we have
\begin{eqnarray*}
	\hat{u}(t)=\frac{M(t)}{\Psi(t)\mathbb{E}\left(M(T)\theta\middle|\mathcal{F}_t\right)}.
\end{eqnarray*}
\begin{proof}
In view of Section 3, the Hamiltonian functional associated to our control problem \eqref{OP} is defined by
\begin{eqnarray}\label{eq:hamiltonian}
\mathcal{H}(t, x, u, \hat{p}(t), \hat{q}) &=& \log(u(t+\delta)) + b_0(t+\delta,t+\delta) x \hat{p}(t+\delta) - u(t+\delta) \hat{p}(t+\delta) + \sigma_0(t+\delta) x \hat{q}(t+\delta,t+\delta)\nonumber\\
&&+ \int_{t+\delta}^T \frac{\partial b_0}{\partial s}(s,t+\delta) x \hat{p}(s)ds.
\end{eqnarray}
Assume that there exists an optimal control $\hat{u} \in \mathcal{A}_{\mathbb{G}}$ associated with the performance functional \eqref{j}, with corresponding state and adjoint processes $(\hat{X}, \hat{p}, \hat{q})$. Then, by the maximum principle, the first-order optimality condition yields, for each $t\in [0,T-\delta]$
\begin{eqnarray*}
	\mathbb{E}\Bigg[\frac{\partial H}{\partial u}\big(t, \hat{X}(t), u, \hat{p}(t), \hat{q}(t, s)\big) \,\big|\, \mathcal{G}_t\Bigg]_{u=\hat{u}(t)} = 0,
\end{eqnarray*}
Consequently,
\begin{eqnarray*}
	\mathbb{E}\left[\frac{1}{\hat{u}(t)} - \hat{p}(t)|\mathcal{G}_t\right] = 0.
\end{eqnarray*}
Since $\hat{u}$ is $\mathbb{G}$-adapted, it follows that
\begin{eqnarray}\label{eq:optimal_u}
	\hat{u}(t) = \frac{1}{\mathbb{E}[\hat{p}(t) \,|\, \mathcal{G}_t]}. 
\end{eqnarray}
For the optimal control $\hat{u}$, the corresponding adjoint process satisfies the following linear advanced backward stochastic Volterra integral equation (ABSVIE): for $0 \leq t \leq T-\delta$.
\begin{eqnarray}\label{ABSVIE}
\hat{p}(t) = \theta + \int_{t+\delta}^{T} \left[b_0(t+\delta,s)\hat{p}(s+\delta)+ \sigma_0(s)\hat{q}(t,s+\delta)\right]ds-\int_t^T \hat{q}(t,s)dB(s). 
\end{eqnarray}
To solve this ABSVIE, we follow the approach of Theorem 3.1 in Hu and Øksendal \cite{HO}. In view of Girsanov theorem, let consider $\widetilde{B}$ a $\mathbb{Q}$-Brownian motion defined by 
\begin{eqnarray*}
	\widetilde{B}(t)= B(t) -\int_0^t \sigma_0(s)ds \quad t \in [0,T].
\end{eqnarray*}
Therefore equation \eqref{ABSVIE} under $\Q$ becomes 
\begin{eqnarray*}
\hat{p}(t) = \theta + \int_{t+\delta}^T b_0(t+\delta,s)\hat{p}(s+\delta)ds -\int_t^T \hat{q}(t,s)d\widetilde{B}(s),
\end{eqnarray*}
which implies, taking the $\Q$-expectation with respect to $\mathcal{F}_t$, that
\begin{eqnarray}\label{Cal}
\hat{p}(t)&=&\E^{\Q}\left[\theta + \int_{t+\delta}^T b_0(t+\delta,s)\hat{p}(s+\delta)ds\middle|\mathcal{F}_t\right]\nonumber\\
&=& F(t)+\int_{t+\delta}^T b_0(t+\delta,s)\E^{\Q}\left[\hat{p}(s+\delta)|\mathcal{F}_t\right]ds,
\end{eqnarray}
where
\begin{eqnarray*}
\tilde{F}(t)=\mathbb{E}_{\mathbb{Q}}[\theta \mid \mathcal{F}_t].
\end{eqnarray*}
Since $b_0$ is deterministic, we look for a solution of the form 
\begin{eqnarray*}
\hat{p}(t)=\varphi(t)F(t),
\end{eqnarray*}
where $\varphi$ is a deterministic function.
Plugging into \eqref{Cal}, we obtain
\begin{eqnarray*}
\varphi(t)F(t)=F(t)+\int_{t+\delta}^T b_0(t+\delta,s)\varphi(t+\delta)\E^{\Q}\left[F(t+\delta)|\mathcal{F}_t\right]ds.
\end{eqnarray*}
Using the tower property
\begin{eqnarray*}
	\E^{\Q}\left[F(s+\delta)|\mathcal{F}_t\right]=F(t), 
\end{eqnarray*}
which yields
\begin{eqnarray*}
\varphi(t)F(t)&=&F(t)+\int_{t+\delta}^T b_0(t+\delta,s)\varphi(s+\delta)F(t)ds\\
.&=& F(t)\left [1+\int_{t+\delta}^T b_0(t+\delta,s)\varphi(s+\delta)ds\right).
\end{eqnarray*}
Therefore, $\varphi$ satisfies the deterministic advanced Volterra equation
\begin{eqnarray*}
	\varphi(t)=1+\int_{t+\delta}^T b_0(t+\delta,s)\varphi(s+\delta)ds.
\end{eqnarray*}
The solution can be written as a Neumann series (resolvent expansion) such that
$\varphi(t)=\Psi(t)$, where $\Psi$ is defined by \eqref{Res}.
Then we get
\begin{eqnarray*}
\hat{p}(t)&=&	\Psi(t)F(t)\\
&=&\Psi(t)\E^{\Q}(\theta|\mathcal{F}_t)\\
&=&\Psi(t)\frac{\E(M(T)\theta|\mathcal{F}_t)}{M(t)}.
\end{eqnarray*}
Finally according to \eqref{eq:optimal_u} we obtain
\begin{eqnarray*}
\hat{u}(t)=\left(\E\left[\Psi(t)\frac{\E(M(T)\theta|\mathcal{F}_t)}{M(t)}|\mathcal{G}_t\right]\right)^{-1}	.
\end{eqnarray*}
In the case $\mathcal{F}=\mathcal{G}$, this reduces to
\begin{eqnarray*}
\hat{u}(t)=\frac{M(t)}{\Psi(t)\E(M(T)\theta|\mathcal{F}_t)}.
\end{eqnarray*}
\end{proof}
\begin{remark}
Suppose that $b_{0}(t,s)=b,\; \sigma_0(t)=\sigma,\; \theta\in L^2(\mathcal{F}_T$ and $\mathcal{F}=\mathcal{G}$. In this case, since
\begin{eqnarray*}
\hat{p}(t)=\Psi(t)\E^{\Q}(\theta|\mathcal{F}_t),	
\end{eqnarray*}
where $\Psi$ is solution of the advanced ordinary differential
\begin{eqnarray*}
	\Psi(t)=1+b\int^T_{t+\delta}\Psi(s+\delta)ds,\;\; \Psi(t)=1, \; t\geq T,
\end{eqnarray*}
we derive
\begin{eqnarray*}
\Psi(t)=\exp(b(T-t-\delta)^{+}),
\end{eqnarray*}
where $x^{+}=\max(x,0)$.
Next, we obtain 
\begin{eqnarray*}
\hat{p}(t)&=&\exp(b(T-t-\delta)^{+})\E^{\Q}(\theta|\mathcal{F}_t)\\
&=& \exp(b(T-t-\delta)^{+})\frac{\E(M(T)\theta|\mathcal{F}_t)}{M(t)},
\end{eqnarray*}
which provides finally
\begin{eqnarray*}
\hat{u}(t)&=& \frac{M(t)}{\exp(b(T-t-\delta)^{+})\E(M(T)\theta|\mathcal{F}_t)}.
\end{eqnarray*}
\end{remark}

\end{document}